\documentclass[11pt,A4]{article}
\usepackage[dvips]{graphicx}
\usepackage{amsmath}
\usepackage{amssymb,color}
\usepackage{enumerate}
\usepackage{epic}
\usepackage{pstricks}
\usepackage{pst-node}
\usepackage{pstricks-add}
\usepackage{multirow}

\def\BBox{\kern  -0.2cm\hbox{\vrule width 0.2cm height 0.2cm}}

\newtheorem{theorem}{Theorem}[section]
\newtheorem{lemma}[theorem]{Lemma}

\newtheorem{definition}[theorem]{Definition}
\newtheorem{corollary}[theorem]{Corollary}
\newtheorem{proposition}[theorem]{Proposition}

\newcommand{\grCfive}[2]{
\psset{linewidth=#1pt}
\SpecialCoor
\degrees[7]
\cnode*(1;1.75){#2pt}{Q71}
\cnode*(1;2.75){#2pt}{Q72}
\cnode*(1;4.75){#2pt}{Q74}
\cnode*(1;5.75){#2pt}{Q75}
\cnode*(1;7.75){#2pt}{Q77}
\ncline{Q71}{Q74}
\ncline{Q71}{Q75}
\ncline{Q72}{Q75}
\ncline[linewidth=0.5pt,linestyle=dotted, dotsep=0.5pt]{Q72}{Q77}
\ncline{Q74}{Q77} }

\newcommand{\grCseven}[2]{
\psset{linewidth=#1pt}
\SpecialCoor
\degrees[7]
\cnode*(1;1.75){#2pt}{Q71}
\cnode*(1;2.75){#2pt}{Q72}
\cnode*(1;3.75){#2pt}{Q73}
\cnode*(1;4.75){#2pt}{Q74}
\cnode*(1;5.75){#2pt}{Q75}
\cnode*(1;6.75){#2pt}{Q76}
\cnode*(1;7.75){#2pt}{Q77}
\ncline{Q71}{Q74}
\ncline{Q71}{Q75}
\ncline{Q72}{Q75}
\ncline{Q72}{Q76}
\ncline{Q73}{Q76}
\ncline{Q73}{Q77}
\ncline{Q74}{Q77}}

\newcommand{\grPfourPtwo}[2]{
\psset{linewidth=#1pt}
\SpecialCoor
\degrees[7]
\cnode*(1;1.75){#2pt}{Q71}
\cnode*(1;2.75){#2pt}{Q72}
\cnode*(1;3.75){#2pt}{Q73}
\cnode*(1;4.75){#2pt}{Q74}
\cnode*(1;5.75){#2pt}{Q75}
\cnode*(1;6.75){#2pt}{Q76}
\cnode*(1;7.75){#2pt}{Q77}
\ncline[linestyle=dashed, dash=2pt 3pt]{Q71}{Q72}
\ncline[linestyle=dashed, dash=2pt 3pt]{Q71}{Q77}
\ncline[linewidth=0.5pt,linestyle=dotted, dotsep=0.5pt]{Q71}{Q73}
\ncline[linewidth=0.5pt,linestyle=dotted, dotsep=0.5pt]{Q71}{Q76}
\ncline[linestyle=dashed, dash=2pt 3pt]{Q73}{Q74}
\ncline[linestyle=dashed, dash=2pt 3pt]{Q75}{Q76} }

\newcommand{\grPtwo}[2]{
\psset{linewidth=#1pt}
\cnode*( 0.7, 0.7){#2pt}{Q42}
\cnode*( 0.7,-0.7){#2pt}{Q43}
\cnode*(-0.7,-0.7){#2pt}{Q44}
\ncline{Q42}{Q44}
\ncline[linestyle=dashed, dash=2pt 3pt]{Q43}{Q44} }

\newcommand{\grTwoKTwo}[2]{
\psset{linewidth=#1pt}
\cnode*(-0.7, 0.7){#2pt}{Q41}
\cnode*( 0.7, 0.7){#2pt}{Q42}
\cnode*( 0.7,-0.7){#2pt}{Q43}
\cnode*(-0.7,-0.7){#2pt}{Q44}
\ncline{Q41}{Q43}
\ncline{Q42}{Q44} }

\newcommand{\grPThree}[2]{
\psset{linewidth=#1pt}
\cnode*(-0.7, 0.7){#2pt}{Q41}
\cnode*( 0.7, 0.7){#2pt}{Q42}
\cnode*( 0.7,-0.7){#2pt}{Q43}
\cnode*(-0.7,-0.7){#2pt}{Q44}
\ncline[linestyle=dashed, dash=2pt 3pt]{Q41}{Q42}
\ncline[linewidth=0.5pt,linestyle=dotted, dotsep=0.5pt]{Q41}{Q44}
\ncline[linewidth=0.5pt,linestyle=dotted, dotsep=0.5pt]{Q42}{Q43} }

\newcommand{\grHoneFive}[2]{
\psset{linewidth=#1pt}
\SpecialCoor
\degrees[5]
\cnode*(1;5.25){#2pt}{Q51}
\cnode*(1;4.25){#2pt}{Q52}
\cnode*(1;2.25){#2pt}{Q54}
\ncline{Q51}{Q54}
\ncline{Q52}{Q54} }

\newcommand{\grGoneFive}[2]{
\psset{linewidth=#1pt}
\SpecialCoor
\degrees[5]
\cnode*(1;1.25){#2pt}{Q50}
\cnode*(1;5.25){#2pt}{Q51}
\cnode*(1;4.25){#2pt}{Q52}
\cnode*(1;3.25){#2pt}{Q53}
\cnode*(1;2.25){#2pt}{Q54}
\ncline{Q50}{Q52}
\ncline{Q50}{Q53}
\ncline{Q51}{Q53}
\ncline{Q51}{Q54}
\ncline{Q52}{Q54} }

\newcommand{\grHtwoFive}[2]{
\psset{linewidth=#1pt}
\SpecialCoor
\degrees[5]
\cnode*(1;1.25){#2pt}{Q50}
\cnode*(1;5.25){#2pt}{Q51}
\cnode*(1;4.25){#2pt}{Q52}
\cnode*(1;2.25){#2pt}{Q54}
\ncline[linestyle=dashed, dash=2pt 3pt]{Q50}{Q51}
\ncline[linestyle=dashed, dash=2pt 3pt]{Q50}{Q54}
\ncline[linestyle=dashed, dash=2pt 3pt]{Q51}{Q52} }

\title{Biregular cages  of girth five.}
\author{
M. Abreu$^1$,
G. Araujo-Pardo$^2$,
C. Balbuena$^3$, \\
D. Labbate$^1$,
G. L\'opez-Ch\'avez$^2$
\\
$^1$ {\em Dipartimento di Matematica}\\
Universit\`a degli Studi della Basilicata,\\
Viale dell'Ateneo Lucano, I-85100 Potenza, Italy \\
marien.abreu@unibas.it; domenico.labbate@unibas.it
\\[1ex]
$^2$ {\em Instituto de Matem\'aticas } \\
Universidad Nacional Aut\'onoma de M\'exico \\
Ciudad Universitaria,
M\'exico D.F. 04510,
M\'exico. \\
garaujo@math.unam.mx; gloria04@gmail.com
\\[1ex]
$^3$ {\em Departament de Matem\`atica Aplicada III}\\
 Universitat Polit\`ecnica de Catalunya \\
 Campus
Nord, Edifici C2,
C/ Jordi Girona 1 i 3\\ E-08034 Barcelona, Spain. \\
m.camino.balbuena@upc.edu
}

\date{}

\begin{document}
\maketitle

\begin{abstract}

Let $2 \le r < m$ and $g$ be positive integers. An \emph{$(\{r,m\};g)$--graph} (or biregular graph) is a graph with degree set $\{r,m\}$ and girth
$g$, and an \emph{$(\{r,m\};g)$--cage} (or biregular cage) is an $(\{r,m\};g)$--graph of minimum order $n(\{r,m\};g)$. If $m=r+1$, an
$(\{r,m\};g)$--cage is said to be a \emph{semiregular cage}.

In this paper we generalize the reduction and graph amalgam operations from \cite{AABL11} 
on the incidence graphs of an affine and a biaffine plane
obtaining two new infinite families of biregular cages and two new semiregular cages.
The constructed new families are $(\{r,2r-3\};5)$--cages for all $r=q+1$ with $q$ a prime power,
and $(\{r,2r-5\};5)$--cages for all $r=q+1$ with $q$ a prime. The new semiregular cages are constructed for $r=5$ and $6$ with $31$ and $43$ vertices respectively.
\end{abstract}


\section {Introduction}\label{intro}

All graphs considered are finite, undirected and simple (without loops or
multiple edges). For definitions and notations not explicitly stated the reader may refer to
\cite{BM}, \cite{B93} and \cite{CL96}.

Let $G $ be a graph with vertex set $V = V (G)$ and edge set $E =
E(G)$. The {\em girth} of a graph $G$ is the length $g = g(G)$ of
a shortest cycle. The {\em degree} of a vertex $v \in V$ is the
number of vertices adjacent to $v$.
A graph is called {\em
$r$--regular} if all its vertices have the same degree $r$.
A {\em $(r,g)$--graph}
is a $r$--regular graph of girth $g$ and a {\em $(r,g)$--cage} is
a $(r,g)$--graph with the smallest possible number of vertices.
Cages have been intensely studied
since they were introduced by Tutte \cite{T47} in $1947$.
Erd\H{o}s and Sachs \cite{ES63} proved the existence of a
$(r,g)$--graph for any value of $r$ and $g$.
Biggs is the author of an impressive report on distinct
methods for constructing cubic cages \cite{B98}.
More details about constructions of cages can be found in
the surveys by Wong \cite{W82}, by Holton and Sheehan
\cite[Chapter 6]{HoSh93}, or the recent one by Exoo and Jajcay
\cite{EJ08}.

The cages theory has been generalized in many ways,
one such is as follows: if $D=\{a_1,\ldots, a_k\}$ is a set of positive integers with $2\le a_1<a_2<\ldots<a_k$ then a $(D;g)$--graph is a graph with degree
set $D$ and girth $g$ and a $(D;g)$--cage is a $(D;g)$--graph with minimum order $n(D;g)=n(a_1,\ldots, a_k;g)$. It is obvious that the $(r; g)$--cage is a special case of the $(D; g)$--cage when $D=\{r\}$.

Few values of $n(D;g)$ are known. In particular, Kapoor et al. \cite{KPW77} proved that $n(D;3)=1+a_k$.
Moreover, the following lower bound for $n(D;g)$ was given by Downs
et al. \cite{DGM81}:
\begin{equation}\label{Dcota}
n(D;g)\ge \left\{\begin{array}{ll } \displaystyle 1+\sum_{i=1}^{t}
a_k(a_1-1)^{i-1}
 \,  &\mbox{if } g=2t+1; \\[2ex]
 \displaystyle 1+\sum_{i=1}^{t-1} a_k(a_1-1)^{i-1}+(a_1-1)^{t-1}\,  &\mbox{if } g=2t.\\[2ex]
\end{array}\right. \end{equation}

A {\em biregular $(\{r,m\};g)$--graph} is a  $(D;g)$--graph with degree set $D=\{r,m\}$ and girth $g$ and
a {\em bi-regular $(\{r,m\};g)$--cage} is an $(\{r,m\};g)$--graph of smallest possible order. Note that bouquets of $r$ cycles of length $g$ are $(\{2,2r\};g)$-cages and the complete bipartite graphs $K_{r,m}$ are $(\{r,m\};4)$--cages.

The existence of biregular $(\{r,m\};g)$--graphs has been proved by  Chartrand, Gould, and Kapoor in \cite{CGK81} for all $2 \le r < m$ and $g \ge 3$ (also proved by F\"{u}redi et al. in \cite{FLSUW95}). On the other hand, several contructions of biregular $(\{r,m\};g)$--cages have been achieved for different values of $r$, $m$ and $g$. In particular, Chartrand et al. proved  in \cite{CGK81} that $n(\{r,m\};4)=r+m$, for $2\le r<m$, and $g=4$ and they also proved in \cite{CGK81} that $n(\{2,m\};g)$ attains the lower bound (\ref{Dcota}). Furthermore, Yuansheng and Liang in \cite{YL03} proved  that $n(\{r,m\};6)=2(rm-m+1)$
for $g=6$ and $r<m$ when $2\le r \le 5$ or $r \ge 2$ and $m-1$ a prime power and  they conjectured in \cite{YL03} that
{\em $n(\{r,m\};6)= 2(rm-m+1)$}, for any $r$ with $2 \le r<m$.

In this paper, we focus our attention on biregular graphs of girth exactly five, where $n(\{r,m\};5)=rm+1$.
Some known results in this case, Downs et al. \cite{DGM81} have shown that $n(\{3; m\}; 5)=3m+1$, for any $m \ge 4$, Hanson et al. \cite{HWJ92} have shown that
$n(\{4; m\}; 5)=4m+1$,  for any integer $m\ge 5$, and Araujo--Pardo et al. in \cite{ABV09} proved several results in this context.
Moreover, it is worth to note that Yuansheng and Liang in \cite{YL03}, claim to have proved that $n(\{5,m\};5)=5m+1$ for $m \ge 6$ in an unpublished manuscript. In Table \ref{exactv}, we summarize these and other known results on the exact values of  $n(\{r,m\};g)$, for $g\ge5$,
together with the results obtained in this paper which we now proceed to describe.

\begin{table}[h]

{\scriptsize
\begin{center}
\renewcommand{\tabcolsep}{1.2mm}
\begin{tabular}{|l|l|c|c|c|c|c|c|c|}\hline
   $r$ & $m$                 & $g=5$        & $g=6$              & $g=7$        & $g=8$                           & $g=9$        & $g=11$      \\
\hline \rule{0pt}{16pt}
   $r=3$ & $m\ge 4$          & $3m+1$       & $4m+2$             & $7m+1$       & $\displaystyle \frac{25m}{3}+5$ & $15m+1$      & $31m+1$     \\
         &                   & \cite{DGM81} & \cite{HWJ92,YL03}  & \cite{DGM81} & $m=3k$ \cite{ABV09}             & $m\ge 6$     & $m=4k$      \\
   \cline{6-6}  &            &              &                    &              & $9m+3$                          & \cite{DGM81} &\cite{ABV09} \\
                &            &              &                    &              & $m=4,5,7$ \cite{ABLM}        &              &             \\
\hline
   $r=4$ & $m\ge 5$      & $4m+1$       & $6m+2$             & $13m+1$      &                                 &              & $121m+1$    \\
         &               &\cite{HWJ92}  &\cite{YL03}         & $m=6k$       &                                 &              & $m=6k$      \\
         &               &              &                    & \cite{ABV09} &                                 &              &\cite{ABV09} \\
\hline \rule{0pt}{14pt}
   $r \ge 5$ & $m=2k(r-1)$  & $1+rm$      &                    &$1+m(r^2-r+1)$&                                 &              &$1+m\frac{(r-1)^5-1}{r-2}$\\
   $r=p^h+1$ &              &\cite{ABV09} &                    &\cite{ABV09}  &                                 &              &\cite{ABV09}  \\
   \cline{2-4}
   $p$ prime & $m=k(r-1)+1$ &             & $2(rm-m+1)$        &              &                                 &              &              \\
             &              &             & \cite{ABGMV08,YL03}&              &                                 &              &              \\
   \cline{2-4}
             &  $m=kr$      &             & \cite{ABV09}       &              &                                 &              &              \\
   \cline{2-4}
             &  $m=2r-3$    &    (*)      &                    &              &                                 &              &              \\
   \cline{1-3}
   $h=1$     &  $m=2r-5$    &    (*)      &                    &              &                                 &              &              \\
   $p \ge 7$ &              &             &                    &              &                                 &              &              \\
\hline
  \end{tabular}

  \end{center}
  }
 \caption{ Exact values of $n(\{r,m\};g)$. The symbol (*) means results obtained in this paper.
 \label{exactv}}
\end{table}

We generalize the reduction and graph amalgam operations from \cite{AABL11} on the incidence graphs $A_q$ of an affine and $B_q$ of a biaffine plane (elliptic semiplane of type $\cal C$) obtaining two new infinite families of biregular $(\{r,m\}; 5)$--cages and two new semiregular cages of girth $5$.
The constructed families are $(\{r,2r-3\};5)$--cages for all $r=q+1$ with $q$ a prime power,
and $(\{r,2r-5\};5)$--cages for all $r=q+1$ with $q$ a prime. The new semiregular cages are constructed for $r=5$ and $6$ with $31$ and $43$ vertices respectively.

The paper is organized as follows: the graphs $A_q$ and $B_q$ are presented in Section \ref{prel}, with a labelling which
will be necessary for the construction of biregular cages.
In Section \ref{Consr2r_3}, we construct previously unknown $(\{r,2r-3\};5)$--cages, for a prime power $q \ge 2$ and the integer $r=q+1$, in Theorem \ref{r-2r-3-5-cage}, adding edges to the graph $A_q$. Moreover if $r-1$ is even, we exhibit $r$ non-isomorphic such $(\{r,2r-3\};5)$--cages in Theorem \ref{non-isomorphic-r-2r-3-5-cage}. In particular, we find a $(\{4,5\};5)$--semiregular cage with $21$ vertices. In Section \ref{Operations} we slightly generalize reduction and amalgam operations, described in \cite{AABL11} and \cite{F10}, that, performed on the bipartite graph $B_q$, will allow us to construct new $(\{r,2r-5\};5)$--cages, for $r=q+1$ with $q \ge 7$ prime, in Section \ref{Consr2r_5}, Theorems \ref{r2r-5c3} and \ref{r2r-5c1}.
Finally, in Section \ref{Small} we construct two new semiregular cages, namely a $(\{5,6\};5)$-cage with $31$ vertices and a $(\{6,7\};5)$-cage with $43$ vertices. Note that the latter is a sporadic example in which we adapt
and slightly generalize the techniques that we have used in Section \ref{Operations}.

\section{Preliminaries}\label{prel}
Let $q=p^n \ge 2$ be a prime power and $\alpha$ a primitive $n^{th}$--root of unity. Consider the finite field
$GF(q) = \{ 0, 1, \alpha, \alpha^2, \ldots, \alpha^{q-2}\}$ and denote $GF^*(q)=GF(q) \setminus \{0\}$.

The graphs constructed in this paper arise from the (bipartite)
incidence graph $B_q$ of an elliptic semiplane of type $C$ (cf. \cite{Demb,AFLN3,F10})
together with the (bipartite) incidence graph $A_q$ of the affine plane of order $q$.
We fix a labelling on their vertices which will be central for our constructions since it
allows us to keep track of the properties (such as regularity and girth) of the graphs obtained from $B_q$ and $A_q$ applying some operations such as
reductions and amalgams (cf. Sections \ref{Consr2r_3}, \ref{Operations}).

\begin{definition}\label{BqHq}

Let $q\ge 2$ be a prime power, and consider the finite field $GF(q)$.

\begin{enumerate}[(i)]
  \item Let $B_q$ be a bipartite graph with vertex set
$(V_0,V_1)$ where $V_r=GF(q)\times GF(q)$, $r=0,1$; and the
edge set defined as follows:
\begin{equation}\label{Bq}
(x,y)_0\in V_0 \mbox{ adjacent to } (m,b)_1 \in V_1  \mbox{ if and only if } y=mx+b.
\end{equation}

  \item Let $A_q$ be the graph obtained from $B_q$ by adding the following set $L_q:= \{(q,x)_1 \,| \, x\in GF(q)\}$ of $q$ vertices and the set $E_q:= \{uv \, | \, u:=(q,x)_1 \,, v:=(x,y)_0 \, \mbox{\em and} \, x,y \in GF(q)\}$ of $q^2$ edges.
\end{enumerate}
\end{definition}

The graph $B_q$ is also known  as the incidence graph of the biaffine
plane \cite{H04} and the graph $A_q$ is the incidence graph of an affine plane of order $q$.
The graph $B_q$ has been used in the problem of finding extremal graphs without short cycles (cf. e.g. \cite{AABL11, AFLN1, AFLN2, ABH10, AB11, AGMS07,  LU95}).

The following properties of the  graph $B_q$  are well known (see
\cite{AABL11,H04,LU95}) and they will be fundamental throughout the paper.

\begin{proposition}\label{BqProp}

Let $B_q$ be the (bipartite) incidence graph  defined above. Let
$P_x=\{(x,y)_0 |\ y\in GF(q)\}$, for $x\in GF(q)$, and $L_m=
\{(m,b)_1 |\ b\in GF(q)\}$, for $m\in GF(q)$. Then the graph $B_q$ has
the following properties:

\begin{enumerate}[(i)]
\item it is $q$--regular, vertex transitive, of order $2q^2$ and
has girth $6$ for $q\ge 3$;
\item it admits a partition
$\displaystyle V_0 =\bigcup_{x\in GF(q)} P_x$ and $\displaystyle V_1 =\bigcup_{m\in GF(q)} L_m$ of its vertex set;

\item each block $P_x$ is connected to each block $L_m$ by a perfect matching, for $x,m \in GF(q)$;

\item each vertex in $P_0$ and $L_0$ is connected {\em straight}
to all its neighbours in $B_q$, meaning that $N((0,y)_0)=\{(i,y)_1
| i \in GF(q) \}$ and   $N((0,b)_1)=\{(j,b)_0 | j \in GF(q) \}$;
\item the other matchings between $P_x$ and $L_m$ are {\em
twisted} and the rule is defined algebraically in $GF(q)$
according to (\ref{Bq}).
\end{enumerate}
\end{proposition}

For further information regarding these properties and for constructions of the adjacency matrix
of $B_q$ as a block $(0,1)$--matrix please refer to \cite{ABL10, AFLN3, B08}.

\section{Construction of a family of $(\{r,2r-3\};5)$--cages.}\label{Consr2r_3}

\noindent In this section, for a prime power $q \ge 2$ and the integer $r=q+1$, we construct previously unknown $(\{r,2r-3\};5)$--cages, adding edges to the graph $A_q$ presented in Section \ref{prel}.

Let $q\ge2$ be a prime power and let $r=q+1$. We define $R_q$ to be the graph with $V(R_q):=V(A_q)$ and $E(R_q):=E(A_q) \cup D$ where
$D= \{(m,0)_1(m,b)_1 \,| \, b\in GF^{*}(q) \mbox{ and } m\in GF(q) \cup \{q\}\}.$

\begin{theorem}\label{r-2r-3-5-cage}
Let $q\ge2$ be a prime power and let $r=q+1$. Then the graph $R_q$ is a $(\{r,2r-3\};5)$--cage satisfying Downs' bound, i.e. $n(\{r,2r-3\};5)=r(2r-3) + 1$.
\end{theorem}

\begin{proof}
The vertices $M:=\{(m,0)_1 | m\in GF(q) \cup \{q\} \} \subset V(R_q)$ have degree $q+(q-1)=2q-1=2r-3$, and the remaining vertices
of $R_q$ have degree $q+1=r$.

By construction $B_q \subset A_q \subset R_q$. Let $C$ be a cycle in $R_q$.
If the edges of $C$ are totally contained in $A_q$ then the length of $C$ is at least six, since $A_q$ is the incidence graph of an affine plane. Otherwise, $C$ contains at least an edge  $e=xy \in D$, where $x,y \in L_m$ and $m \in GF(q) \cup \{q\}$. Then it follows from the bipartition and girth of $A_q$ that the distance $d_{A_q}(x,y)=4$. Thus the length of $C$ in this case is at least five and exactly five if $C$ contains exactly one edge of $D$. Hence, the graph $R_q$ has girth $5$ since $C=(q,0)_1(q,1)_1(1,0)_0(0,0)_1(0,0)_0(q,0)_1$ where $(q,0)_1(q,1)_1$ is the only edge of $C$ in $D$.

Finally, $|V(R_q)|=|V(A_q)|=2q^{2}+q=2r^{2}-3r+1=r(2r-3)+1$.\end{proof}

\begin{corollary}\label{r,r+1case}
The graph $R_q$ is a semi--regular cage if and only if $r=4$.
\end{corollary}

\begin{proof}
It follows immediately since $2r-3=r+1$ if and only if $r=4$. \end{proof}

Note that an isomorphic graph to $R_4$ has been found also in \cite{HWJ92} .

\noindent In a similar way, for $q$ even, we construct a family of non--isomorphic  $(\{r,2r-3\};5)$--cages.

Let $q$ be an even prime power and let $D_m := \{(m,0)_1(m,b)_1 \, | \, b \in GF(q)\}$ and
$F_m := \{ (m,0)_1(m,1)_1 \} \cup \{(m, \alpha^i)_1(m, \alpha^{i+1})_1 \, | \, 1 \le i \le q-3, \, i \, \mbox{odd}\}$, for $m \in GF(q) \cup \{q\}$.
Then $D_m  \cong K_{1,q-1}$ is a star with vertex set $L_m$ and $F_m$ is a matching between the vertices of $L_m$.

Let $0 \le t \le q-1$ and let $I_t=\{\alpha^{q-t-1}, \ldots , \alpha^{q-2}\}$ be a set of indexes.
We define $G_t$ to be the graph with $V(G_t):=V(A_q)$ and $E(G_t):=E(A_q) \cup F_t \cup D_t$, where
$$F_t:=\bigcup_{m \, \in \, GF(q) \setminus I_t}F_m \quad \mbox{ and } \quad D_t:=\bigcup_{m \, \in \,  I_t \cup \{q\}}D_m.$$

Note that the graph $G_t$ is obtained from the graph $A_q$ adding $t+1$ stars and $q-t$ matchings within the sets $L_m$.
In particular, for $t=0$ the index set $I_t= \emptyset$ and the only star added to $A_q$ is the one in $L_q$.
Moreover, if we set $I_q:=GF(q)$, then we can say that the graph $G_q:=R_q$.

\begin{theorem}\label{non-isomorphic-r-2r-3-5-cage}
Let $q=2^{s}$ be an even prime power, with $s\ge1$. Then there are at least $q+1$ non--isomorphic $(\{r,2r-3\};5)$--cages.
\end{theorem}

\begin{proof}
Let $G_t$ be the family of graphs defined as above, for $0\le t\leq q$.
Reasoning as in Theorem \ref{r-2r-3-5-cage}, is follows that, for every $0 \le t \le q-1$, the graphs $G_t$ have girth $5$, order $r(2r-3)+1$ and
are biregular with $t+1$ vertices of degree $2r-3$. Thus, $G_t$ is a $(\{r,2r-3\};5)$--cage for every $0 \le t \le q-1$.
Moreover, $G_i \ncong G_j$, for $i,j \in GF(q) \cup \{q\}$ with $i \ne j$, since they have a different number of vertices of degree $2r-3$.\end{proof}

\section{Operations on $B_q$} \label{Operations}

In this section we slightly generalize reduction and amalgam operations, described in  \cite{AABL11} and  \cite{F10},
that, performed on the bipartite graph $B_q$, will allow us to construct new $(\{r,2r-5\};5)$--cages, for $r=q+1$ with $q \ge 7$ prime, in Section \ref{Consr2r_5}.

\subsection{Reductions}\label{reduction}

We describe two reduction operations on $B_q$ already introduced in \cite{AABL11}.
The first one is exactly the same while the second one is slightly generalized.

\noindent{\sc Reduction 1}\cite{AABL11}  Remove vertices from $P_0$ and $L_0$.

Let $T \subseteq S \subseteq GF(q)$, $S_0 = \{(0,y)_0 | y \in S\}
\subseteq P_0$, $T_0 = \{(0,b)_1 | b \in T\}  \subseteq L_0$ and
$B_q(S,T)=B_q-S_0-T_0$.

\begin{lemma}\label{Red1}
Let $T \subseteq S \subseteq GF(q)$. Then $B_q(S,T)$ is
biregular with degrees $(q-1,q)$ of order $2q^2-|S|-|T|$.
Moreover, the vertices $(i,t)_0 \in V_0$ and $(j,s)_1 \in V_1$,
for each $i,j \in GF(q) -\{0\}$, $s \in S$ and $t \in T$ are the
only vertices of degree $q-1$ in $B_q(S,T)$, together with
$(0,s)_1 \in V_1$ for $s \in S-T$ if $T \subsetneq S$.
\end{lemma}

\begin{proof}
It is an immediate consequence of Proposition \ref{BqProp} $(i)$, $(v)$.
\end{proof}

\noindent{\sc Reduction 2} Remove blocks $P_i$ and $L_j$ from $B_q$ or from $B_q(S,T)$.

\noindent

Let $u_0,u_1$ be non--negative integers such that $0 \le u_0 \le u_1 < q-1$. If $u_i > 0$, let $U_i := \{ \alpha^{q-j} \in GF(q) :  j=2, \ldots, u_i+1 \}$ be an index set, for $i=0,1$. Let $\mathcal U_0$ and $\mathcal U_1$ be sets of blocks of $B_q$ chosen as follows:
$$
\mathcal U_0 := \{P_x\subset V_0: x \in U_0\}  \mbox{ if }  1 \le u_0 \le q-1
\quad \quad \mbox{ or } \quad \quad  \mathcal U_0 := \emptyset \mbox{ if } u_0=0
$$
$$
\mathcal U_1 := \{L_m\subset V_1: m \in U_1\}  \mbox{ if }  1 \le u_1 \le q-1
\quad \quad \mbox{ or } \quad \quad  \mathcal U_1 := \emptyset \mbox{ if } u_1=0.
$$
Then, we define $B_q(u_0,u_1):=B_q- \mathcal U_0-\mathcal U_1$ to be the graph obtained from $B_q$ by deleting the last $u_0$ blocks
of $V_0$, and the last $u_1$ blocks of $V_1$. Analogously, we define $B_q(S,T,u_0,u_1):=B_q-(S_0\cup T_0)-(\mathcal U_0\cup \mathcal U_1)$
to be the graph obtained from the graph $B_q(S,T)$ obtained with Reduction 1. Clearly, for $u_0=u_1=0$, $B_q(0,0)=B_q$ and $B_q(S,T,0,0)=B_q(S,T)$.

\begin{lemma}\label{Red2}
Let  $u_0,u_1$ be non--negative integers, with $0 \le u_0 \le u_1 < q-1$. Then
\begin{enumerate}[(i)]
\item the graph $B_q(u_0,u_1)$ is a biregular graph with degrees $\{q-u_0,q-u_1\}$ of order $2q^2-q(u_0+u_1)$ if $u_0 \ne u_1$ ;

\item the graph $B_q(u_0,u_1)$ is a $(q-u_0)$--regular graph of order $2q(q-u_0)$ if $u_0=u_1$ ;

\item the graph $B_q(S,T,u_0,u_1)$ has degrees $\{q-u_0,q-u_1,q-u_0-1,q-u_1-1\}$ and order $2q^{2}-q(u_0+u_1)-|S|-|T|$.
Moreover, the vertices $(i,t)_0\in V_0$ with  $i\in GF(q)-U_0$ and $t\in T$ are the only vertices of degree $q-u_1-1$ in $B_q(S,T,u_0,u_1)$
and the vertices $(j,s)_1\in V_1$ with $j\in GF(q)-U_1$ and $s\in S$, together with the vertices $(0,s)_1 \in V_1$, for $s \in S-T$ if $T \subsetneq S$,
are the only vertices of degree $q-u_0-1$ in $B_q(S,T,u_0,u_1)$;
\end{enumerate}
\end{lemma}

\begin{proof}
It is an immediate consequence of  Proposition \ref{BqProp} $(i)$, $(v)$ and Lemma \ref{Red1}.
\end{proof}

\subsection{Amalgam}\label{amal}

In this section we describe an {\em amalgam} operation inspired by J{\o}rgensen \cite{J05}, Funk \cite{F10} and Abreu et al. \cite{AABL11},
where regular bipartite graphs were transformed into (no longer bipartite) regular graphs of higher degree adding
{\em weighted} edges with different weights on opposite sides of the bipartition.

Since we apply Reduction $1$ before increasing the degree of $B_q$,
we describe the amalgam operation performed on the reduced graph $B_q(S,T,u_0,u_1)$ for $0 \le u_0 \le u_1 < q-1$.
The labelling for $B_q$ introduced in Section \ref{prel}, will be essential, in the choice of the graphs
used for the amalgam, to guarantee the biregularity and the girth 5 in the final graph.

Let $\Gamma_1$ and $\Gamma_2$ be two graphs of the same order and with the same labels on their vertices.
In general, an {\em amalgam of $\Gamma_1$ into $\Gamma_2$} is a graph obtained adding all the edges of
$\Gamma_1$ to $\Gamma_2$.

Let $P_i$ and $L_i$ be defined as in Section \ref{prel}.
Consider the graph $B_q(S,T,u_0,u_1)$,  for some $T \subseteq S \subseteq GF(q)$ and $0 \le u_0 \le u_1 < q-1$.
Let $S_0 \subseteq P_0$, $T_0 \subseteq L_0$ as in Reduction $1$, and let $P_0':=P_0-S_0$ and $L_0':=L_0-T_0$ be
the blocks in $B_q(S,T,u_0,u_1)$ of order $q-|S|$ and $q-|T|$, respectively.

Let $H_1$, $H_2$, $G_i$, for $i=1,2$, be graphs of girth at least $5$ and order $q-|S|$, $q-|T|$ and $q$, respectively.
To simplify notation in our results, we label $P_i$ and $L_i$ as in Section \ref{prel},
but assume that the labellings of $H_1,H_2,G_1$ and $G_2$ correspond
to the second coordinates of $P_0', L_0', P_i$ and $L_j$ respectively for $i \in GF^*(q)-U_0$
and $j \in GF^*(q)-U_1$.

We define $B^*_q(S,T,u_0,u_1)$ to be the {\em amalgam} of $H_1$ into $P_0'$,
$H_2$ into $L_0'$, $G_1$ into $P_i$, for $i \in GF^*(q)-U_0$, and $G_2$ into $L_j$, for $j \in GF^*(q)-U_1$. Note that $|V(B^*_q(S,T,u_0,u_1))|=|V(B_q(S,T,u_0,u_1))|$.

The next lemma is immediate and it shows the behavior of the degree set of the graph $B^*_q(S,T,u_0,u_1)$.

\begin{lemma}\label{rgrados}
Let $G:=B^{*}_q(S,T,u_0,u_1)$. Then the degrees of the vertices of $G$ are:
\begin{enumerate}
\item $d_G((0,y)_0) = q - u_1 + d_{H_1}(y)$,
\item $d_G((0,b)_1) = q - u_0 + d_{H_2}(b)$ if $b \notin S-T$,
\item $d_G((0,b)_1) = q - u_0 - 1 + d_{H_2}(b)$ if $b \in S-T$,
\item $d_G((x,y)_0) = q - u_1 + d_{G_1}(y)$ if $y \notin T$,
\item $d_G((x,y)_0) = q - u_1 - 1 + d_{G_1}(y)$ if $y \in T$,
\item $d_G((m,b)_1) = q - u_0 + d_{G_2}(b)$ if $b \notin S$,
\item $d_G((m,b)_1) = q - u_0 - 1 + d_{G_2}(b)$ if $b \in S$,
\end{enumerate}
\end{lemma}

Note that with the above mentioned labelling, the labels of $G_1,G_2$ are the elements of $GF(q)$
and the labels of $H_1,H_2$ are either the elements of $GF(q)$ itself or a proper subset according to
$S,T$ being empty or not.

Let $M_F:=\{(u,v) : u,v \in GF(q) \mbox{ and } uv \in E(F)\}$, for $F \in \{H_1,H_2,G_1,G_2\}$.
For each $(u,v) \in M_F$, we define $\omega ((u,v)) = \pm (u-v) \in GF^*(q)$ to be its {\em weight} or {\em Cayley Colour}.
We define $\Omega(F):=\{ \omega((u,v)) : (u,v) \in M_F\}$ to be the {\em set of weights} or {\em set of Cayley Colours} of $F$, for $F \in \{H_1,H_2,G_1,G_2\}$.

\noindent Note that $\Omega(F_1) \cap \Omega(F_2) = \emptyset$ implies that $M_{F_1} \cap M_{F_2} = \emptyset$, for $F_1, F_2 \in \{H_1,H_2,G_1,G_2\}$ and $F_1 \ne F_2$, but the converse is false.

The following lemma generalizes theorems \cite[Theorem 5]{AABL11}  and \cite[Theorem 2.8]{F10}.

\begin{theorem}\label{mainth}
Let $T \subseteq S \subseteq GF(q)$ and let $0\le u_0\le u_1 <q-1$. Let $H_1,H_2,G_1$ and $G_2$ be defined as above and suppose that
$M_{H_1} \cap M_{H_2} = \emptyset$, $M_{H_1} \cap M_{G_2} = \emptyset$, $M_{H_2} \cap M_{G_1} = \emptyset$
and $\Omega(G_1) \cap \Omega(G_2) = \emptyset$. Then the amalgam $B^*_q(S,T,u_0,u_1)$ has girth at least $5$ and order $2q^2-q(u_0+u_1)-|S|-|T|$.
\end{theorem}

\begin{proof}
Suppose first that $u_0=u_1=0$ and so $B^*_q(S,T)=B^*_q(S,T,u_0,u_1)$ (cf. Reduction 2).

Let $C$ be a shortest cycle in $B^*_q(S,T)$ and suppose, by
contradiction, that $|C| \leq 4$. Therefore, $C=(xyz)$ or
$C=(wxyz)$. Since $B_q$ has girth $6$ and $H_1,H_2,G_1,G_2$ have
girth at least $5$, then $C$ cannot be completely contained in
$B_q$ or in some $H_i$ or $G_i$ for $i=1,2$. Then, w.l.o.g. the
path $xyz$ in $C$ is such that $x,y \in P_i$ and $z \in L_m$ for
some $i,m \in GF(q)$. Since the edges between $P_i$ and $L_m$ form
a matching, then $xz \notin E(B_q)$ and hence $xz \notin E(B_q^*(S,T))$.
Thus $|C| > 3$ and we can assume $|C|=4$ and $C=(wxyz)$,with $xyz$ taken as before.

If $w \in P_i$, by the same argument, $wz \notin E(B_q^*(S,T))$
and we have a contradiction. There are no edges between $P_i$ and
$P_j$ in $B^*_q(S,T)$, so $w \notin P_j$ for $j \in GF(q) - \{i\}$,
which implies that $w \in L_n$ for some $n \in GF(q)$.
If $n \ne m$, we have a contradiction since there are no edges
between $L_m$ and $L_n$ in $B^*_q(S,T)$. Therefore $x,y \in P_i$
and $w,z \in L_m$. Let $x=(i,a_1)_0, y=(i,a_2)_0,
w=(m,b_1)_1$ and $z=(m,b_2)_1$ as in the labelling chosen in
Section \ref{prel}. Then $wx,yz \in E(B^*_q(S,T))$ imply that
$a_1=m \cdot i + b_1$ and $a_2=m \cdot i + b_2$,
respectively.

If $m$ or $i$ are zero, i.e. if $xy \in H_1$ or $wz \in H_2$,
then the above equations are satisfied if and only if
$a_1 = b_1$ and $a_2 = b_2$, but this contradicts at least one of
$M_{H_1} \cap M_{H_2} = \emptyset$, $M_{H_1} \cap M_{G_2} = \emptyset$ and
$M_{H_2} \cap M_{G_1} = \emptyset$.

If $m$ and $i$ are both non-zero, i.e. if $xy \in G_1$ and $wz \in G_2$,
then the above equations are satisfied if and only if
$a_1 - a_2 = b_1 - b_2$, implying that $\pm (a_1 - a_2) \in \Omega(G_1)$
and $\pm (a_1 - a_2) = \pm (b_1 - b_2) \in \Omega(G_2)$
which contradicts $\Omega(G_1) \cap \Omega(G_2) = \emptyset$.

Hence $B^*_q(S,T)$ has girth at least five. Since $B^*_q(S,T,u_0,u_1)$ is a subgraph of $B^*_q(S,T)$, for $0\le u_0\le u_1 <q-1$, then also $B^*_q(S,T,u_0,u_1)$ ha girth five, completing the proof.
\end{proof}

\section{New $(\{r,2r-5\};5)$-cages for $r \ge 8$.}\label{Consr2r_5}

In this section we will construct a new family of  $(\{r,2r-5\};5)$-cages, for $r=q+1$ and $q\ge 7$ a prime, applying Reduction 1, 2 and Amalgam (cf. Theorem \ref{mainth} and Lemma \ref{rgrados}) to the graph $B_q$ as previously described.

Recall that every prime $q$ is either congruent to $1$ or $3$ modulo $4$, we will now treat these two cases separately, when $q=4n+1$ or $q=4n+3$.
In each case we will specify the sets $S$ and $T$ to be deleted from $P_0$ and $L_0$, the integers $u_0$ and $u_1$ of number of blocks to be deleted, and the graphs $H_1$, $H_2$, $G_1$ and $G_2$ to be used for the amalgam into $B_q^{*}(S,T,u_0,u_1)$.

Since $q$ is a prime we can consider that $GF(q)$ coincides with $\mathbb Z_q$, and the addition operations are modulo $q$.
We present the case $q \equiv 3 \mod 4$ first, since the smallest case of the construction occurs for $q=7 \equiv 3 \mod 4.$
In what follows, recall that $D(F)$ denotes the degree set of a graph $F$.

\subsection{Construction for primes $q=4n+3$, $n \ge 1$.}

Let $B_q^*(S,T,u_0,u_1)$ be the graph resulting from  the following choice of its parameters:

\noindent
\resizebox{15cm}{!}{
\begin{minipage}[h]{20cm}
\begin{tabular}{|l|c|c|l|}
\hline
\multicolumn{4}{|c|}{\rule{0pt}{20pt} $S = \{\frac{q+1}{4}, - \frac{q+1}{4}\} = \{\frac{q+1}{4}, \frac{3q-1}{4}\};$ \hspace{1cm}  
$T = \emptyset;$ \hspace{1cm} $u_0=0;$ \hspace{1cm}  $u_1=1$ }\\
\multicolumn{4}{|c|}{} \\
\hline
Graph & Vertices & Edges & Description \\
\hline \rule{0pt}{20pt}
$H_1$ & $\mathbb Z_q-\{\frac{q+1}{4}, \frac{3q-1}{4}\}$ &
$\left\{ (j,j+\frac{q-1}{2}) \, | \,  j\in \mathbb Z_q-\{\frac{q+1}{4}, \frac{3q-1}{4}, \frac{3q+3}{4} \} \right\} \cup
\left\{ (\frac{3q+3}{4},\frac{q-3}{4}) \right\}$ &
$(q-2)$--cycle \\
& & sums modulo $q$ & $\Omega(H_1) = \{\frac{q-1}{2}, \frac{q-3}{2} \}$ \\
\hline \rule{0pt}{20pt}
$G_1$ & $\mathbb Z_q$ &
$\left\{ (j,j+\frac{q-1}{2}) \, | \,  j\in \mathbb Z_q \right\}$ &
$q$--cycle \\
& & sums modulo $q$ & $\Omega(G_1) = \{\frac{q-1}{2} \}$ \\
\hline \rule{0pt}{20pt}
$H_2 \cong G_2$ & $\mathbb Z_q$ &
$\left\{(0,j) : j\in \mathbb Z_q^*-\{\frac{q-1}{2},\frac{q+1}{2}\} \right\} \cup
\left\{ (\frac{q+1}{4},\frac{q-1}{2}), (\frac{3q-1}{4},\frac{q+1}{2})\} \right\}$ & $\Omega(H_2) = \Omega(G_2) = \mathbb Z_q^* - \{\frac{q-1}{2} \}$ \\
\hline
\end{tabular}
\end{minipage}
}

\

To illustrate the construction we present in Figure \ref{Fig1_8_11_5} the graph $B_q^{*}(S,T,u_0,u_1)$ without the edges from $B_q$, for $q=7$.
Each line style represents a different weight (or Cayley Colour).
As we will proved in the Theorem \ref{r2r-5c3}, this is an $(\{8,11\};5)$--cage.

\begin{figure}[h!]
\begin{center}
\begin{tabular}{ccccccl}
  \rule{0.5cm}{0pt}
  \begin{pspicture}(-.7,-.7)(.7,.7)
  \psset{unit=0.6}
  \grCfive{0.6}{2.6}
  \end{pspicture} &
  \begin{pspicture}(-.7,-.7)(.7,.7)
  \psset{unit=0.6}
  \grCseven{0.6}{2.6}
  \end{pspicture} &
  \begin{pspicture}(-.7,-.7)(.7,.7)
  \psset{unit=0.6}
  \grCseven{0.6}{2.6}
  \end{pspicture} &
  \begin{pspicture}(-.7,-.7)(.7,.7)
  \psset{unit=0.6}
  \grCseven{0.6}{2.6}
  \end{pspicture} &
  \begin{pspicture}(-.7,-.7)(.7,.7)
  \psset{unit=0.6}
  \grCseven{0.6}{2.6}
  \end{pspicture} &
  \begin{pspicture}(-.7,-.7)(.7,.7)
  \psset{unit=0.6}
  \grCseven{0.6}{2.6}
  \end{pspicture} &
  \begin{pspicture}(-.7,-.7)(.7,.7)
  \psset{unit=0.6}
  \grCseven{0.6}{2.6}
  \end{pspicture} \\
  \rule{0.5cm}{0pt}
  \begin{pspicture}(-.7,-.7)(.7,.7)
  \psset{unit=0.6}
  \grPfourPtwo{0.6}{2.6}
  \end{pspicture} &
  \begin{pspicture}(-.7,-.7)(.7,.7)
  \psset{unit=0.6}
  \grPfourPtwo{0.6}{2.6}
  \end{pspicture} &
  \begin{pspicture}(-.7,-.7)(.7,.7)
  \psset{unit=0.6}
  \grPfourPtwo{0.6}{2.6}
  \end{pspicture} &
  \begin{pspicture}(-.7,-.7)(.7,.7)
  \psset{unit=0.6}
  \grPfourPtwo{0.6}{2.6}
  \end{pspicture} &
  \begin{pspicture}(-.7,-.7)(.7,.7)
  \psset{unit=0.6}
  \grPfourPtwo{0.6}{2.6}
  \end{pspicture} & \begin{pspicture}(-.7,-.7)(.7,.7)
  \psset{unit=0.6}
  \grPfourPtwo{0.6}{2.6}
  \end{pspicture}  &
  \begin{pspicture}(-0.7,-0.7)(3,0.7)
  \psset{unit=0.6}
  \psline[linestyle=dashed, dash=2pt 3pt](-1,1)(0,1) \rput(1,1){$\omega = 1$}
  \psline[linewidth=0.5pt,linestyle=dotted, dotsep=0.5pt](-1,0)(0,0) \rput(2,0){$\omega = \frac{q-3}{2} = 2$}
  \psline(-1,-1)(0,-1) \rput(2,-1){$\omega = \frac{q-1}{2} = 3$}
  \end{pspicture} \\
\end{tabular}
\caption{$B^{*}_{7}(S,T,0,1)-E(B_7)$ with $S=\{3,5\}$ and $T=\emptyset$ }\label{Fig1_8_11_5}
\end{center}
\end{figure}
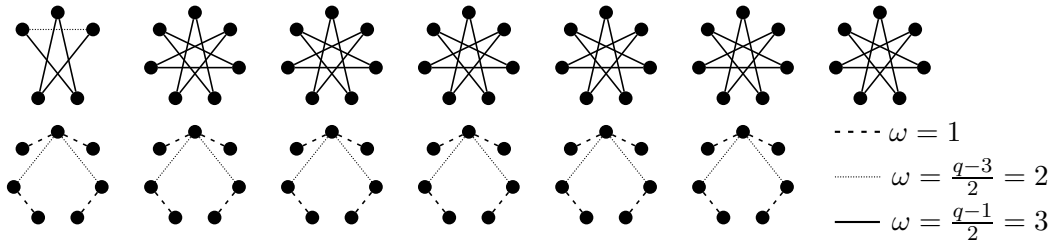

 \begin{theorem}\label{r2r-5c3}
 Let $q=4n+3$ be a prime, for $n \ge 1 $. Let $S,T,u_0,u_1,H_1,H_2,G_1$ and $G_2$ be defined as above.
 Then the amalgam graph $B^*_q(S,T,u_0,u_1)$ is an $(\{r,2r-5\};5)$--cage of order $r(2r-5)+1$, where $r=q+1$.
 \end{theorem}

 \begin{proof}
 From Lemma \ref{rgrados}, the degree set $D(B_q^{*}(S,T,u_0,u_1))=\{r,2r-5\}$, where $r=q+1$.
 In fact, all vertices in $V_0 \cap B_q(S,T,u_0,u_1)$ have degree $q-1$. Moreover, $D(H_1)=D(G_1)=\{2\}$
 since $H_1$ and $G_1$ are cycles. Hence, all the vertices of $V_0 \cap V(B_q^{*}(S,T,u_0,u_1))$ have degree $q+1=r$.
 In $V_1$ we distinguish three subsets of vertices: $V_1':=\{(m,0)_1: m\in \mathbb Z_q\}$,
 $V_1'':=\{(m,t)_1: m \in \mathbb Z_q, t \in \frac{q+1}{4}, \frac{3q-1}{4} \}$ and $V_1''':=\{(m,t)_1: t \in \mathbb Z_q^* - \{\frac{q+1}{4}, \frac{3q-1}{4}\}\}$.
 The vertices in $V_1'''$  have degree $q$ in $B_q(S,T,u_0,u_1)$, and their degree is $1$ in $H_2$ and $G_2$. Thus these vertices
 have degree $q+1=r$ in $B_q^{*}(S,T,u_0,u_1)$. The vertices of $V_1''$ have degree $q-1$ in  $B_q(S,T,u_0,u_1)$, whereas their degree is  $2$
 in $H_2$ and $G_2$. Hence, they have degree $q+1=r$ in $B_q^{*}(S,T,u_0,u_1)$.
 Finally, the vertices in $V_1'$ have degree $q$ in $B_q(S,T,u_0,u_1)$ , while they have degree  $q-3$ in $H_2$ and $G_2$.
 Hence, they have degree $q+q-3=2q-3=2(r-1)-3=2r-5$ in $B_q^{*}(S,T,u_0,u_1)$.

The graph $B_q^{*}(S,T,u_0,u_1)$ has girth at least $5$, since we are under the hypothesis of Theorem \ref{mainth}.
In fact, $\Omega(H_1) \cap \Omega(H_2) = \Omega(H_1) \cap \Omega(G_2) = \{\frac{q+3}{2}\}$, but
$M_{H_1} \cap M_{H_2} = M_{H_1} \cap M_{G_2} = \emptyset$, since the only edge of weight $\frac{q+3}{2}$
in $H_1$ is $(\frac{3q+3}{4},\frac{q-3}{4})$, while  the edges $(0, \frac{q-3}{2})$
$(0, \frac{q+3}{2})$ are those of weight $\pm \frac{q+3}{2}$ in $H_2$ and $G_2$.
We also have that $\Omega(G_1) \cap \Omega(H_2) = \Omega(G_1) \cap \Omega(G_2) = \emptyset$, which
implies that $M_{G_1} \cap M_{H_2} = M_{G_1} \cap M_{G_2} = \emptyset$.
Moreover, the girth is exactly five, since the $5$--cycle $((0,0)_0,(0,\frac{q+1}{4})_0,(0,\frac{q-1}{2})_0, (0,\frac{q-1}{2})_1,(0,0)_1)$ lies in  $B_q^{*}(S,T,u_0,u_1)$.

Finally, the order of $B_q^{*}(S,T,u_0,u_1)$ is $2q^{2}-q(u_0+u_1)-|S|-|T|=2q^{2}-q-2=r(2r-5)+1$, by Lemma \ref{Red2},
which satisfies exactly Down's bound (\ref{Dcota})
$$
 n(\{r,2r-5\};5) = 1 + \sum_{i=1}^{2} (2r-5)(r-1)^{i-1}=r(2r-5)+1.
$$
Hence, the graph $B_q^{*}(S,T,0,1)$ is an  $(\{r,2r-5\};5)$-cage.
 \end{proof}

\subsection{Construction for primes $q=4n+1$, $n \ge 3$.}
Let $B_q^*(S,T,u_0,u_1)$ be the graph resulting from  the following choice of its parameters:

\noindent
\resizebox{15cm}{!}{
\begin{minipage}[h]{20cm}
\begin{tabular}{|l|c|c|l|}
\hline
\multicolumn{4}{|c|}{\rule{0pt}{20pt} $S = \{\frac{q-1}{4}, - \frac{q-1}{4}\} = \{\frac{q-1}{4}, \frac{3q+1}{4}\};$ \hspace{1cm}  
$T = \emptyset;$ \hspace{1cm} $u_0=0;$ \hspace{1cm}  $u_1=1$ }\\
\multicolumn{4}{|c|}{} \\
\hline
Graph & Vertices & Edges & Description \\
\hline \rule{0pt}{20pt}
$H_1$ & $\mathbb Z_q-\{\frac{q-1}{4}, \frac{3q+1}{4}\}$ &
$\left\{ (j,j+\frac{q-1}{2}) \, | \,  j\in \mathbb Z_q-\{\frac{q-1}{4}, \frac{3q+1}{4}, \frac{3q-3}{4} \} \right\} \cup
\left\{ (\frac{3q-3}{4},\frac{q+3}{4}) \right\}$ &
$(q-2)$--cycle \\
& & sums modulo $q$ & $\Omega(H_1) = \{\frac{q-1}{2}, \frac{q+3}{2} \}$ \\
\hline \rule{0pt}{20pt}
$G_1$ & $\mathbb Z_q$ &
$\left\{ (j,j+\frac{q-1}{2}) \, | \,  j\in \mathbb Z_q \right\}$ &
$q$--cycle \\
& & sums modulo $q$ & $\Omega(G_1) = \{\frac{q-1}{2} \}$ \\
\hline \rule{0pt}{20pt}
$H_2 \cong G_2$ & $\mathbb Z_q$ &
$\left\{(0,j) : j\in \mathbb Z_q^*-\{\frac{q-1}{2},\frac{q+1}{2}\} \right\} \cup
\left\{ (\frac{q-1}{4},\frac{q-1}{2}), (\frac{3q+1}{4},\frac{q+1}{2})\} \right\}$ & $\Omega(H_2) = \Omega(G_2) = \mathbb Z_q^* - \{\frac{q-1}{2} \}$ \\
\hline
\end{tabular}
\end{minipage}
}

 \begin{theorem}\label{r2r-5c1}
 Let $q=4n+1$ be a prime, for $n \ge 3 $. Let $S,T,u_0,u_1,H_1,H_2,G_1$ and $G_2$ be defined as above.
 Then the amalgam graph $B^*_q(S,T,u_0,u_1)$ is an $(\{r,2r-5\};5)$--cage of order $r(2r-5)+1$, where $r=q+1$.
 \end{theorem}

\begin{proof}
Analogous to the proof of Theorem \ref{r2r-5c3}.
\end{proof}

\section{New semiregular $(\{r,r+1\};5)$-cages for $r=5,6.$}\label{Small}

In this section we construct two new semiregular cages, namely a $(\{5,6\};5)$-cage with $31$ vertices and a $(\{6,7\};5)$-cage with $43$ vertices.
For the first one, we choose $S,T,u_0,u_1,H_1,H_2,G_1$ and $G_2$ to construct the amalgam graph $B_q^{*}(S,T,u_0,u_1)$ and the result is obtained
as a consequence of Theorem \ref{mainth} and Lemma \ref{rgrados}. The second one is a sporadic example in which we adapt
and slightly generalize the techniques that we have used so far.

\subsection{Construction of the $(\{5,6\};5)$-cage.}
Let $GF(4)=\{0,1,\alpha,\alpha^{2}\}$ be the finite field of order $4$ and let $B_q^*(S,T,u_0,u_1)$ be the graph resulting from  the following choice of its parameters
(c.f. Figure \ref{Fig2_5_6_5}):

 \noindent
\resizebox{20cm}{!}{
\begin{minipage}[h]{20cm}
\begin{tabular}{|l|c|c|l|}
\hline
\multicolumn{4}{|c|}{\rule{0pt}{16pt} $S = \{0\};$ \hspace{1cm}
$T = \emptyset;$ \hspace{1cm} $u_0=0;$ \hspace{1cm}  $u_1=0$ }\\[8pt]
\hline
Graph & Vertices & Edges & Description \\
\hline \rule{0pt}{20pt}
$H_1$ & $GF^*(4)$ &
$\left\{ (1,\alpha^{2}),(\alpha^{2},\alpha) \right\}$ & $2$--path \\
 & & & $\Omega(H_1) = \{1, \alpha \}$ \\
\hline \rule{0pt}{20pt}
$G_1$ & $GF(4)$ &
$\left\{ (0,\alpha),(1,\alpha^{2}) \right\}$ & Two disjoint edges \\
 & & & $\Omega(G_1) = \{\alpha \}$ \\
\hline \rule{0pt}{20pt}
$H_2 \cong G_2$ & $GF(4)$ &
$\left\{(\alpha^{2},0), (0,1), (1,\alpha) \right\}$ & $3$--path \\
 & & & $\Omega(H_2) = \Omega(G_2) = \{1, \alpha^2 \}$ \\
\hline
\end{tabular}
\end{minipage}
}

\begin{theorem}\label{56}
Let $q=4$ and let $S,T,u_0,u_1,H_1,H_2,G_1$ and $G_2$ be defined as above.
Then the amalgam graph $B^*_4(S,T,u_0,u_1)$ is a $(\{5,6\};5)$--cage of order $31$.
 \end{theorem}

\begin{proof} From Lemma \ref{rgrados} we have that the degree set $D({B_4^{*}(S,T,u_0,u_1)})=\{5,6\}$, and moreover, the set of vertices of degree $6$ is
$\{(i,1)_0: i\in GF(4)\}\cup \{(0,\alpha^{2})_1\}$.

The graph $B_4^{*}(S,T,u_0,u_1)$ has girth at least $5$, since we are under the hypothesis of Theorem \ref{mainth}.
In fact, $\Omega(H_1) \cap \Omega(H_2)$ $= \Omega(H_1) \cap \Omega(G_2)$ $= \{1\}$, but
$M_{H_1} \cap M_{H_2} = M_{H_1} \cap M_{G_2} = \emptyset$, since $(\alpha^2,\alpha)$ is the only edge of weight $1$
in $H_1$, while the only edges of such weight is $(0, 1)$ in $H_2$ and $G_2$.
We also have that $\Omega(G_1) \cap \Omega(H_2) = \Omega(G_1) \cap \Omega(G_2) = \emptyset$, which
implies that $M_{G_1} \cap M_{H_2} = M_{G_1} \cap M_{G_2} = \emptyset$.
Moreover, the girth is exactly five, since the $5$--cycle $((0,0)_0,(0,1)_0,(0,\alpha)_0, (1,\alpha)_1,(1,0)_1)$ lies in  $B_4^{*}(S,T,u_0,u_1)$.

 Finally, by Lemma \ref{Red2} we have that the order of the graph is $2q^{2}-q(u_0+u_1)-|S|-|T|=2(16)-1=31$.
 \end{proof}

\

\begin{figure}[h!]
\begin{center}
\begin{tabular}{ccccl}
  \begin{pspicture}(-1,-1)(1,1)
  \psset{unit=0.6}
  \grPtwo{0.6}{2.6}
  \rput( 1.1, 1.1){\tiny$(0,1)_0$}
  \rput( 1.1,-1.1){\tiny$(0,\alpha)_0$}
  \rput(-1.1,-1.1){\tiny$(0,\alpha^2)_0$}
  \end{pspicture} &
  \begin{pspicture}(-1,-1)(1,1)
  \psset{unit=0.6}
  \grTwoKTwo{0.6}{2.6}
  \rput(-1.1, 1.1){\tiny$(1,0)_0$}
  \rput( 1.1, 1.1){\tiny$(1,1)_0$}
  \rput( 1.1,-1.1){\tiny$(1,\alpha)_0$}
  \rput(-1.1,-1.1){\tiny$(1,\alpha^2)_0$}
  \end{pspicture} &
  \begin{pspicture}(-1,-1)(1,1)
  \psset{unit=0.6}
  \grTwoKTwo{0.6}{2.6}
  \rput(-1.1, 1.1){\tiny$(\alpha,0)_0$}
  \rput( 1.1, 1.1){\tiny$(\alpha,1)_0$}
  \rput( 1.1,-1.1){\tiny$(\alpha,\alpha)_0$}
  \rput(-1.1,-1.1){\tiny$(\alpha,\alpha^2)_0$}
  \end{pspicture} &
  \begin{pspicture}(-1,-1)(1,1)
  \psset{unit=0.6}
  \grTwoKTwo{0.6}{2.6}
  \rput(-1.1, 1.1){\tiny$(\alpha^2,0)_0$}
  \rput( 1.1, 1.1){\tiny$(\alpha^2,1)_0$}
  \rput( 1.1,-1.1){\tiny$(\alpha^2,\alpha)_0$}
  \rput(-1.1,-1.1){\tiny$(\alpha^2,\alpha^2)_0$}
  \end{pspicture} &
  \multirow{2}{1.8cm}{
  \begin{pspicture}(-0.7,-0.7)(3,0)
  \psset{unit=0.6}
  \psline[linestyle=dashed, dash=2pt 3pt](-1,1)(0,1) \rput(1.5,1){$\omega = 1$}
  \psline(-1,0)(0,0) \rput(1.6,-1){$\omega = \alpha^2$}
  \psline[linewidth=0.5pt,linestyle=dotted, dotsep=0.5pt](-1,-1)(0,-1) \rput(1.5,0){$\omega = \alpha$}
  \end{pspicture}
 }
 \\
  \begin{pspicture}(-1,-1)(1,1)
  \psset{unit=0.6}
  \grPThree{0.6}{2.6}
  \rput(-1.1, 1.1){\tiny$(0,0)_1$}
  \rput( 1.1, 1.1){\tiny$(0,1)_1$}
  \rput( 1.1,-1.1){\tiny$(0,\alpha)_1$}
  \rput(-1.1,-1.1){\tiny$(0,\alpha^2)_1$}
  \end{pspicture} &
  \begin{pspicture}(-1,-1)(1,1)
  \psset{unit=0.6}
  \grPThree{0.6}{2.6}
  \rput(-1.1, 1.1){\tiny$(1,0)_1$}
  \rput( 1.1, 1.1){\tiny$(1,1)_1$}
  \rput( 1.1,-1.1){\tiny$(1,\alpha)_1$}
  \rput(-1.1,-1.1){\tiny$(1,\alpha^2)_1$}
  \end{pspicture} &
  \begin{pspicture}(-1,-1)(1,1)
  \psset{unit=0.6}
  \grPThree{0.6}{2.6}
  \rput(-1.1, 1.1){\tiny$(\alpha,0)_1$}
  \rput( 1.1, 1.1){\tiny$(\alpha,1)_1$}
  \rput( 1.1,-1.1){\tiny$(\alpha,\alpha)_1$}
  \rput(-1.1,-1.1){\tiny$(\alpha,\alpha^2)_1$}
  \end{pspicture} &
  \begin{pspicture}(-1,-1)(1,1)
  \psset{unit=0.6}
  \grPThree{0.6}{2.6}
  \rput(-1.1, 1.1){\tiny$(\alpha^2,0)_1$}
  \rput( 1.1, 1.1){\tiny$(\alpha^2,1)_1$}
  \rput( 1.1,-1.1){\tiny$(\alpha^2,\alpha)_1$}
  \rput(-1.1,-1.1){\tiny$(\alpha^2,\alpha^2)_1$}
  \end{pspicture} &
  \\
\end{tabular}
\caption{$B^{*}_{4}(S,T,0,0)-E(B_4)$ with $S=\{0\}$ and $T=\emptyset$ }\label{Fig2_5_6_5}
\end{center}
\end{figure}
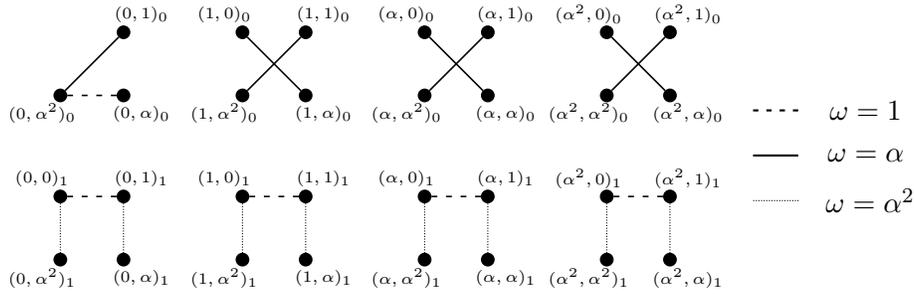

\subsection{Construction of the $(\{6,7\};5)$-cage.}

For this construction we need to modify Reduction $1$ and apply the amalgam operation accordingly as follows:

Let $q=5$ and consider the graph $B_5$. We modify Reduction 1 removing vertices from $P_0$ and vertices of $V_1$ from each block $L_j$ for $j \in GF(5)$ (whereas in Reduction 1 we delete vertices of $V_1$ only from the block $L_0$).

Let $T:=\{3\} \subset GF(5)$, $S:=\{0,3\}\subset GF(5)$, $S_0 = \{(0,y)_0 | y \in S\} \subseteq P_0$,
$T_j = \{(j,b)_1 | b \in T\}  \subseteq L_j$ for $j \in GF(5)$, and let $\displaystyle B_5(S,TT):=B_5-S_0-\bigcup_{j \in GF(5)}T_j$.
Note that the graph $B_5(S,TT)$ has degree set $D(B_5(S,TT))=\{q-1,q\}=\{4,5\}$ and order $2q^2-|S|-q|T|=50-2-5=43$.
Moreover, each vertex in $V_0-P_0$ has degree $q-|T|=4$, each vertex of $V_1$ with second coordinate in $S-T$ has degree
$q-|S-T|=4$, and all other vertices have degree $q=5$.

    Now we will amalgam some graphs into $B_5(S,TT)$. Let $H_1:=\{(2,4),(4,1)\}$, $G_1:=\{(0,2),(2,4),(4,1),(1,3),(3,0)\}$
    and $H_2:=\{(4,0),(0,1),(1,2)\}$. These graphs are a $2$--path, a $5$--cycle and a $3$--path, respectively, with weights
    $\Omega(H_1)=\Omega(G_1)=\{2\}$ and $\Omega(H_2)=\{1\}$.

    Let $B^*_5(S,TT)$ be the graph obtained from the amalgam of
    $H_1$ into $P'_0:=P_0-S_0$, $G_1$ into $P_i$, for all $i \in GF^*(5)$, and $H_2$ into $L'_j=L_j-T_j$, for all $j \in GF(5)$
    (c.f. Figure \ref{Fig3_6_7_5} for an illustration).

    \begin{theorem}\label{67}
Let $q=5$ and let $S,T,H_1,H_2$ and $G_1$ be defined as above.
Then the amalgam graph $B^*_5(S,TT)$ is a $(\{6,7\};5)$--cage of order $43$.
 \end{theorem}

\begin{proof}
Using the same reasoning as in the proof of Theorem \ref{mainth}, the amalgam graph $B^*_5(S,TT)$ has girth at least $5$, since $\Omega(H_1) \cap \Omega(H_2) = \Omega(G_1) \cap \Omega(H_2) = \emptyset$. The girth is exactly $5$, since $G_1$ is a $5$--cycle.
Moreover, the degree set $D(B^*_5(S,TT))=\{6,7\}$, since all vertices of $B_5(S,TT)$ of degree $4$ obtain two new edges in $B^*_5(S,TT)$, and similarly the vertices of $B_5(S,TT)$ of degree $5$  obtain one or two new edges in $B^*_5(S,TT)$.
Hence, $B^*_5(S,TT)$ is a $(\{6,7\};5)$--cage as desired, since its order satisfies Down's bound.
\end{proof}

\begin{figure}[h!]
\begin{center}
\begin{tabular}{ccccc}
  \rule{0.5cm}{0pt}
  \begin{pspicture}(-1,-1)(1.5,1)
  \psset{unit=0.6}
  \grHoneFive{0.6}{2.6}
  \rput(1.8;5.25){\tiny $(0,1)_0$}
  \rput(1.6;4.25){\tiny $(0,2)_0$}
  \rput(1.8;2.25){\tiny $(0,4)_0$}
  \end{pspicture} &
  \begin{pspicture}(-1,-1)(1.5,1)
  \psset{unit=0.6}
  \grGoneFive{0.6}{2.6}
  \rput(1.4;1.25){\tiny $(1,0)_0$}
  \rput(1.8;5.25){\tiny $(1,1)_0$}
  \rput(1.6;4.25){\tiny $(1,2)_0$}
  \rput(1.6;3.25){\tiny $(1,3)_0$}
  \rput(1.8;2.25){\tiny $(1,4)_0$}
  \end{pspicture} &
  \begin{pspicture}(-1,-1)(1.5,1)
  \psset{unit=0.6}
  \grGoneFive{0.6}{2.6}
  \rput(1.4;1.25){\tiny $(2,0)_0$}
  \rput(1.8;5.25){\tiny $(2,1)_0$}
  \rput(1.6;4.25){\tiny $(2,2)_0$}
  \rput(1.6;3.25){\tiny $(2,3)_0$}
  \rput(1.8;2.25){\tiny $(2,4)_0$}
  \end{pspicture} &
  \begin{pspicture}(-1,-1)(1.5,1)
  \psset{unit=0.6}
  \grGoneFive{0.6}{2.6}
  \rput(1.4;1.25){\tiny $(3,0)_0$}
  \rput(1.8;5.25){\tiny $(3,1)_0$}
  \rput(1.6;4.25){\tiny $(3,2)_0$}
  \rput(1.6;3.25){\tiny $(3,3)_0$}
  \rput(1.8;2.25){\tiny $(3,4)_0$}
  \end{pspicture} &
  \begin{pspicture}(-1,-1)(1.5,1)
  \psset{unit=0.6}
  \grGoneFive{0.6}{2.6}
  \rput(1.4;1.25){\tiny $(4,0)_0$}
  \rput(1.8;5.25){\tiny $(4,1)_0$}
  \rput(1.6;4.25){\tiny $(4,2)_0$}
  \rput(1.6;3.25){\tiny $(4,3)_0$}
  \rput(1.8;2.25){\tiny $(4,4)_0$}
  \end{pspicture} 
 \\
  \rule{0.5cm}{0pt}
  \begin{pspicture}(-1,-1)(1.5,1)
  \psset{unit=0.6}
  \grHtwoFive{0.6}{2.6}
  \rput(1.4;1.25){\tiny $(0,0)_1$}
  \rput(1.8;5.25){\tiny $(0,1)_1$}
  \rput(1.6;4.25){\tiny $(0,2)_1$}
  \rput(1.8;2.25){\tiny $(0,4)_1$}
  \end{pspicture} &
  \begin{pspicture}(-1,-1)(1.5,1)
  \psset{unit=0.6}
  \grHtwoFive{0.6}{2.6}
  \rput(1.4;1.25){\tiny $(1,0)_1$}
  \rput(1.8;5.25){\tiny $(1,1)_1$}
  \rput(1.6;4.25){\tiny $(1,2)_1$}
  \rput(1.8;2.25){\tiny $(1,4)_1$}
  \end{pspicture} &
  \begin{pspicture}(-1,-1)(1.5,1)
  \psset{unit=0.6}
  \grHtwoFive{0.6}{2.6}
  \rput(1.4;1.25){\tiny $(2,0)_1$}
  \rput(1.8;5.25){\tiny $(2,1)_1$}
  \rput(1.6;4.25){\tiny $(2,2)_1$}
  \rput(1.8;2.25){\tiny $(2,4)_1$}
  \end{pspicture} &
  \begin{pspicture}(-1,-1)(1.5,1)
  \psset{unit=0.6}
  \grHtwoFive{0.6}{2.6}
  \rput(1.4;1.25){\tiny $(3,0)_1$}
  \rput(1.8;5.25){\tiny $(3,1)_1$}
  \rput(1.6;4.25){\tiny $(3,2)_1$}
  \rput(1.8;2.25){\tiny $(3,4)_1$}
  \end{pspicture} &
  \begin{pspicture}(-1,-1)(1.5,1)
  \psset{unit=0.6}
  \grHtwoFive{0.6}{2.6}
  \rput(1.4;1.25){\tiny $(4,0)_1$}
  \rput(1.8;5.25){\tiny $(4,1)_1$}
  \rput(1.6;4.25){\tiny $(4,2)_1$}
  \rput(1.8;2.25){\tiny $(4,4)_1$}
  \end{pspicture}
  \\
\end{tabular}
\caption{$B^{*}_{5}(S,TT)-E(B_5)$ with $S=\{0,3\}$ and $T=\{3\}$ }\label{Fig3_6_7_5}
\end{center}
\end{figure}
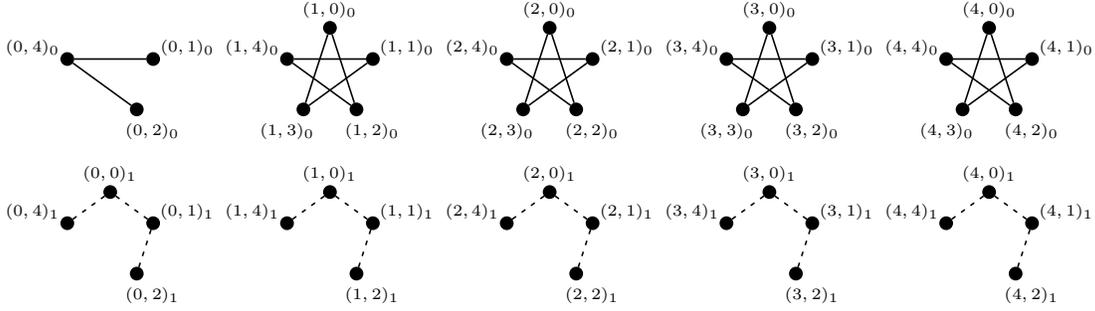

\subsection*{Acknowledgment}
{ Research   supported by the Ministerio de Educaci\'on y Ciencia,
Spain, and the European Regional Development Fund (ERDF) under
project MTM2008-06620-C03-02,  CONACyT-M\'exico under project 57371 and PAPIIT-M\'exico under project 104609-3.
}


\end{document}